\documentclass[reqno,12pt]{preprint}
\usepackage[latin1]{inputenc}
\usepackage[marginparwidth=1in]{geometry}
\usepackage[full]{textcomp}
\usepackage[osf]{newtxtext}
\usepackage{cabin}
\usepackage{enumerate}
\usepackage{enumitem}
\usepackage{bbm}
\usepackage{shuffle}
\usepackage[varqu,varl]{inconsolata}

\usepackage{comment}
\usepackage{hyperref}
\usepackage{breakurl}
\usepackage{mathrsfs}
\usepackage{booktabs}
\usepackage{caption}
\usepackage{stmaryrd}

\usepackage{tikz}
\usepackage{amsmath} 
\usepackage{mhequ} 
\usepackage{mhsymb} 
\usepackage{mhenvs} 
\usepackage{microtype}
\usepackage{amssymb} 
\usepackage{eufrak}
\usepackage{shuffle}

\usepackage{wasysym}
\usepackage{centernot}

\usepackage{nicefrac}

\usepackage{csquotes}

\usepackage{commath}

\newcommand{\D}{\mathbb{D}}

\newcommand{\1}{\mathbbm{1}}

\newcommand{\cB}{\mathcal{B}}

\newcommand{\cH}{\mathcal{H}}

\newcommand{\cN}{\mathcal{N}}
\newcommand{\cP}{\mathcal{P}}

\newcommand{\cS}{\mathcal{S}}

\newcommand{\cW}{\mathcal{W}}

\renewcommand{\R}{\mathbb{R}}

\renewcommand{\D}{\mathbb{D}}
\renewcommand{\N}{\mathbb{N}}

\renewcommand{\E}{\mathbb{E}}
\renewcommand{\P}{\mathbb{P}}

\def\var#1{#1\textnormal{-var}}

\DeclareFontFamily{U}{matha}{\hyphenchar\font45}
\DeclareFontShape{U}{matha}{m}{n}{
	<5> <6> <7> <8> <9> <10> gen * matha
	<10.95> matha10 <12> <14.4> <17.28> <20.74> <24.88> matha12
}{}
\DeclareSymbolFont{matha}{U}{matha}{m}{n}
\DeclareFontSubstitution{U}{matha}{m}{n}
\DeclareFontFamily{U}{mathx}{\hyphenchar\font45}
\DeclareFontShape{U}{mathx}{m}{n}{
	<5> <6> <7> <8> <9> <10>
	<10.95> <12> <14.4> <17.28> <20.74> <24.88>
	mathx10
}{}
\DeclareSymbolFont{mathx}{U}{mathx}{m}{n}
\DeclareFontSubstitution{U}{mathx}{m}{n}
\DeclareMathDelimiter{\vvvert}{0}{matha}{"7E}{mathx}{"17}

\colorlet{darkblue}{blue!90!black}
\colorlet{darkred}{red!90!black}
\colorlet{dr}{red!90!black}

\tikzset{
	dot/.style={thin,circle,fill=gray,draw=black,inner sep=0pt,minimum size=2mm},
	vertex/.style={thin,circle,fill=black,draw=black,inner sep=0pt,minimum size=1mm},
}

\makeatletter 
\newcommand*{\bigcdot}{}
\DeclareRobustCommand*{\bigcdot}{%
	\mathbin{\mathpalette\bigcdot@{}}%
}
\newcommand*{\bigcdot@scalefactor}{.5}
\newcommand*{\bigcdot@widthfactor}{1.15}
\newcommand*{\bigcdot@}[2]{%
	\sbox0{$#1\vcenter{}$}
	\sbox2{$#1\cdot\m@th$}%
	\hbox to \bigcdot@widthfactor\wd2{%
		\hfil
		\raise\ht0\hbox{%
			\scalebox{\bigcdot@scalefactor}{%
				\lower\ht0\hbox{$#1\bullet\m@th$}%
			}%
		}%
		\hfil
	}%
}
\makeatother

\def\dash{\leavevmode\unskip\kern0.18em--\penalty\exhyphenpenalty\kern0.18em}
\def\slash{\leavevmode\unskip\kern0.15em/\penalty\exhyphenpenalty\kern0.15em}

\def\w{\mathtt{w}}
\def\i{\mathtt{i}}
\def\j{\mathtt{j}}
\def\1{\mathtt{1}}
\def\2{\mathtt{2}}
\def\d{\mathtt{d}}
\def\var{\text{\tiny var}}

\begin{document}
	
\title{Gaussian Rough Paths Lifts \\ via Complementary Young Regularity}
\author{Paul Gassiat$^1$ and Tom Klose$^2$}

\institute{CEREMADE, Universit\'{e} Paris Dauphine, PSL University, France \and
	University of Oxford, United Kingdom\\
	\email{gassiat@ceremade.dauphine.fr, tom.klose@maths.ox.ac.uk}}

\maketitle

\begin{abstract}
	Inspired by recent advances in singular SPDE theory, we use the Poincaré inequality on Wiener space to show that controlled complementary Young regularity is sufficient to obtain Gaussian rough paths lifts. This allows us to completely bypass assumptions on the $2$D variation regularity of the covariance and, as a consequence, we obtain cleaner proofs of approximation statements (with optimal convergence rates) and show the convergence of random Fourier series in rough paths metrics under mi\-ni\-mal assumptions on the coefficients (which are sharper than those in the existent literature).
	
	\vspace{1em}
	
	
\end{abstract}

\setcounter{tocdepth}{2}    
\tableofcontents

\section{Introduction}

\noindent
Since Lyons's original work~\cite{Lyons} in 1998, the study of rough paths has grown into a fully fledged theory with countless theoretical and practical applications, see the monographs~\cite{friz-victoir-book} and~\cite{FrizHairer20}. 
Among all paths, \emph{Gaussian} paths form a particularly rich class of examples and the search for conditions under which they can be lifted to the space of rough paths has shaped the theory from its very beginning. 

In this article, we revisit this question from a new perspective that is based on the \emph{Poincaré inequality on Wiener space}.
This idea has originated in the more complicated context of singular~SPDEs, more precisely in the multi-index setting of Linares, Otto, Tempelmayr, and Tsatsoulis~\cite{LOTT24}, and has then allowed to obtain a more comprehensive proof of the BPHZ theorem within tree-based regularity structures, see Hairer and Steele~\cite{Hairer_Steele_BPHZ}.
For a worked example in the form of the generalised KPZ equation, see also the recent work of Bailleul and Bruned~\cite{Bailleul_Bruned_KPZ}. 

In our setting, the benefit of this approach is that it reduces the question of fin\-ding a Gaussian rough path lift to checking~\emph{controlled Complementary Young Regularity}~(cCYR) of the Cameron--Martin space; see Theorem~\ref{thm:cCYR1}, the main finding of this work.
This condition is arguably simpler to verify than conditions on the \mbox{$2$D $\rho$-variation} of the covariance associated to the underlying Gaussian process {that appear in the literature}; at the same time, it is only marginally stronger an assumption than CYR which, for most applications of Gaussian rough paths such as a support theorem~\cite[Sec.~15.8]{friz-victoir-book} or Malliavin differentiability~\cite[Sec.~11.3]{FrizHairer20}, is imposed anyway.

In applications, cCYR is readily implied by the Besov--variation embedding due to Friz and Victoir~\cite{FV06} and, in turn, allows us to obtain optimal convergence rates for piecewise linear approximations. When specialised to fractional Brownian motion with Hurst parameter~$H > 1/4$, this recovers the results of Friz and Riedel~\cite{FR14}. 

Finally, our result implies the convergence of random Fourier series in rough path distances under minimal and, in particular, less restrictive conditions than in the existent literature, see Friz, Gess, Gulisashvili, and Riedel~\cite{FGGR16}.
{As such, random Fourier series provide a concrete example in which our result leads to a stronger statement under weaker assumptions that are also easier to check in practice.}

\section{Preliminaries and main result} \label{sec_ccyr}

\noindent
In this section, we first introduce the necessary notation alongside some preliminary results which we then use to state and prove our main result.

\subsection{Notation and preliminary results} \label{sec_notation}

\noindent
For any~$l, n \in \N_0$ with~$l \leq n$, we set~$\llbracket l,n \rrbracket := \{l,\ldots,n\}$.
For~$\alpha,\beta \in \R$, we write~$\alpha \lesssim \beta$ if~$\alpha \leq C \beta$ for some~$C \in \R$ and~$\alpha \lesssim_{\theta} \beta$ if~$C = C(\theta)$; in estimates across multiple lines, this constant may change from line to line.
Whenever we write~$\alpha^-$, we mean that a statement holds for any~$\gamma < \alpha$, usually having in mind~$\gamma = \alpha - \kappa$ for~$0 < \kappa \ll 1$.

\paragraph*{Gaussian analysis.}
We follow~\cite[Sec.~11.1]{FrizHairer20} and consider a centred Gaussian process~$X = (X^\1, \ldots, X^\d): [0,T] \to \R^\d$ with continuous sample paths, which we realise as~$X(\omega) = \omega \in \Omega := C([0,T];\R^\d)$. 
We then equip~$\Omega$ with the uniform norm and define $\mathcal{F} := \cB(\Omega)$ as well as~$\P := \operatorname{Law}(X)$.
\begin{itemize}[wide, labelwidth=!, labelindent=0pt, parsep=3pt]
	\item The \emph{Cameron--Martin} (C--M) space~$\cH \subseteq \Omega \equiv C([0,T];\R^\d)$ of~$X$ is given by
	\begin{equation*}
		\cH := \{h \in C([0,T];\R^\d): \exists Z \in \cW_1 \ \text{s.t.} \ h_t = \E\sbr[0]{Z X_t} \ \text{for all} \ t \in [0,T]\}
	\end{equation*}
	where~$\cW_1$ is the \emph{first Wiener chaos}, i.e. the $L^2(\P)$-closure of $(X_t^\i: t \in [0,T], \ \i \in \llbracket \1, \d \rrbracket)$.
	\item 
	The map~$W: \cH \to \cW_1$ given by~$W(h) := Z$ for~$h_\cdot = \E\sbr[0]{Z X_\cdot}$ defines an~\emph{isonormal Gaussian process} on~$\cH$, i.e.~a family~$W = \{W(h)\}_{h \in \cH}$ of centered Gaussian random variables on the probability space~$(\Omega,\mathcal{F},\P)$ with covariance~$\E\sbr[0]{W(h)W(h')} =: \langle h,h' \rangle_{\cH}$.
	Equipped with this scalar product, the C--M space~$\cH$ becomes a separable Hilbert space.
	\item Let~$H_n$ denote the $n$-th Hermite polynomial. 
	The \emph{$n$-th homogeneous Wiener chaos}~$\cW_n$ is defined as the $L^2(\P)$-closure of~$\operatorname{span}\{H_n(W(h)): \norm[0]{h}_\cH = 1\}$; we also set~$\cP_n := \oplus_{k=0}^n \cW_k$, the \emph{$n$-the inhomogeneous chaos}.
	\item \emph{Hypercontractivity:} For~$n \in \N$, $p,q \in [1,\infty)$, and~$Y \in \cP_n$, we have the bound $\norm[0]{Y}_{L^q} \lesssim_{p,q,n} \norm[0]{Y}_{L^p}$; see~\cite[Thm.~5.10]{janson}.
	\item Let~$\cS$ be the set of all the cylindrical random variables
	\begin{equation*}
		F = g(W(h_1),\ldots,W(h_m)), \quad h_1, \ldots, h_m \in \cH, \quad g \in C_c^\infty(\R^m; \R), \quad m \geq 1.
	\end{equation*}
	The \emph{Malliavin derivative}~$DF \in L^2(\Omega;\cH)$ of~$F$ w.r.t. $W$ is defined as
	\begin{equation*}
		DF := \sum_{\ell=1}^m (\partial_\ell g)(W(h_1),\ldots,W(h_m)) h_\ell
	\end{equation*}
	\item \emph{Poincaré inequality on Wiener space:} Let~$\D^{1,2}$ be the closure of~$\cS$ under the norm~$\norm[0]{\cdot}_{1,2}$ given by~$\norm[0]{F}^{2}_{1,2} := \E[F^2] + \E\sbr[0]{\norm[0]{DF}_{\cH}^2}$.
	For any~$F \in \D^{1,2}$, we have the bound $\operatorname{Var}[F] \leq \E\sbr[0]{\norm[0]{DF}_{\cH}^2}$,
	see~\cite[Thm.~5.5.1,~Eq.~(5.5.2)]{Bogachev}.
	In particular, this implies
	\begin{equation*}
		\norm[0]{F}_{L^2(\P)} \leq \norm[1]{\norm[0]{DF}_{\cH}}_{L^2(\P)} + \abs[0]{E\sbr[0]{F}}.
	\end{equation*}
\end{itemize}
For further details on Gaussian analysis, we refer the reader to the monographs~\cite{janson} and~\cite{nualart}, or~\cite[App.~D]{friz-victoir-book} for a concise summary.

\paragraph*{Signatures and rough paths.}

\begin{itemize}[wide, labelwidth=!, labelindent=0pt]
	\item Let~$k \geq 1$. A \emph{word}~$\w = \i_1 \ldots \i_k$ with letters~$\i_\ell \in \llbracket \1,\d \rrbracket$ for~$\ell \in \llbracket 1,k \rrbracket$ has \emph{length}~$\abs{\w} := k$.
	For~$(e_k)_{k=\1}^\d$ the standard basis of~$\R^\d$, we write~$e_{\w} := \otimes_{\ell=1}^k e_{\i_\ell} \in (\R^\d)^{\otimes k}$.
	\item Let~$x: [0,T] \to \R^\d$ be a smooth path.
	We define the components of its \emph{truncated $n$-level signature} $S_{0,T}(x) \equiv S_{0,T}^{(n)}(x): [0,T] \to \oplus_{k=0}^n (\R^\d)^{\otimes k}$ by
	\begin{equation*}
		\langle S_{0,T}(x),\w \rangle 
		:=
		\langle S_{0,T}(x),e_\w \rangle
		:= 
		\int_{0 < u_1 < \ldots < u_k < T} \dif x^{\i_1}_{u_1} \ldots \dif x_{u_k}^{\i_k}, 
		\quad
		\w = \i_1 \ldots \i_k, \quad k \leq n.
	\end{equation*}	
	\item We define
	\begin{equation*}
		\threebars x \threebars_{\alpha,n} := \sup_{|\w| \leq n} \sup_{s<t \in [0,T]} \frac{|\left\langle S(x)_{s,t}, \w \right\rangle|}{|t-s|^{|\w| \alpha}}
	\end{equation*}
	and, by abuse of notation, also write~$\threebars S(x) \threebars_{\alpha,n} := \threebars x \threebars_{\alpha,n}$. This is the usual \emph{rough path norm} if $n = \lceil \alpha^{-1} \rceil$.
	\item For $0< \beta \leq \alpha$, we will use the following \emph{inhomogeneous distance}:
	\begin{equation*}
		\rho_{\beta,\alpha,n}[ x , y] := \sup_{|\w| \leq n} \sup_{s<t \in [0,T]} \frac{|\left\langle S(x)_{s,t} - S(y)_{s,t}  , \w \right\rangle|}{|t-s|^{\beta + \alpha(|\w|-1)}}.
	\end{equation*}
	By abuse of notation, we also write~$\rho_{\beta,\alpha,n}[S(x),S(y)] := \rho_{\beta,\alpha,n}[ x , y]$ which, for $\beta = \alpha$ and $n = \lceil \alpha^{-1} \rceil$, is the usual inhomogeneous rough path distance. In that case, we will simply write $\rho_{\alpha}$.
	The reason why we distinguish between~$\alpha$ and~$\beta$ is given in Remark~\ref{rmk_alphabeta}.
	\item We define the space~$\mathcal{C}^\alpha_g([0,T]; \R^\d)$ of \emph{geometric $\alpha$-Hölder rough paths} as the closure of $\{S^{\lceil \alpha^{-1} \rceil}(x): x \in C^\infty([0,T];\R^d)\}$ under the distance~$\rho_\alpha$.
	\item For~$t \in [0,T]$ and a word~$\w$ with~$\abs[0]{\w} \in \N$,  we have~$\langle S(x)_{0,t}, \w \rangle \in \cP_{\abs[0]{\w}}$ by~\cite[Prop.~15.19]{friz-victoir-book}
\end{itemize}

\noindent

We now collect several facts on the distance~$\rho_{\beta,\alpha,n}$. They are well-known when $\alpha = \beta$ (e.g. \cite{friz-victoir-book}) and the proofs for the general case are similar, so we do not give them here.

Recall that if $x$ is a $\alpha$-H\"older rough path, and $h$ is a path of finite $q$-variation with $1/q + \alpha > 1$, we can lift $(x,h)$ to a rough path in a canonical way (see~\cite[Condition~$20.1$]{friz-victoir-book} and the comments preceding it). 
We will need this fact in the following form.

\begin{proposition} \label{prop:YoungTr}
	Let $x:[0,T] \to \R^\d$ be a smooth path, then for any $q \geq 1$ with $1/q + \alpha > 1$, there exists $C=C(\alpha,p,n)$ s.t. for any $\i_1,\ldots, \i_n \in \llbracket \1,\d \rrbracket$, any $j =1,\ldots, n+1$, it holds that
	\begin{equation} \label{eq:YngTr}
		\int_{0 \leq r_{n+1} \leq \ldots \leq r_1 \leq T} \dif x^{\i_1}_{r_1} \ldots \dif h_{r_j} \ldots \dif x^{\i_n}_{r_{n+1}} \leq C \left( 1+ \threebars x \threebars_{\alpha,n}^{n}\right) T^{\alpha n} \left\| h \right\|_{q-\var;[0,T]}.
	\end{equation}
	
	For two smooth paths $x,y$, and any $0 \leq \beta \leq \alpha$ with $\beta + 1/q >1$, we have that
	\begin{align} \label{eq:YngTrDiff}
		& \int_{0 \leq r_{n+1} \leq \ldots \leq r_1 \leq T} \dif x^{\i_1}_{r_1} \ldots \dif h_{r_j} \ldots \dif x^{\i_n}_{r_{n+1}} - \int_{0 \leq r_{n+1} \leq \ldots \leq r_1 \leq T} \dif y^{i_1}_{r_1} \ldots \dif h_{r_j} \ldots \dif y^{i_n}_{r_{n+1}}   \nonumber \\
		& \leq C T^{\alpha (n-1) + \beta}  \left( 1+ \threebars x \threebars_{\alpha,n}^{ n-1}+ \threebars y \threebars_{\alpha,n}^{ n-1} \right)  \left\| h \right\|_{q-\var;[0,T]} \rho_{\beta,\alpha,n}[ x , y] .
	\end{align}
\end{proposition}

We will also need the following Kolmogorov-type Lemma.

\begin{proposition} \label{prop:Kolm}
	Let $X =(X_{t})_{ t \in [0,T]}$ be a random process with paths of finite variation. Then, for any $p \geq 1$,  $\alpha, \beta > 0$ with $\alpha' < \alpha - \frac{1}{p}$, any $n \geq 1$, there exists $C_1 >0$, s.t.
	\begin{equation} \label{eq:K1}
		\sup_{0\leq s< t \leq T} \sup_{|\w|\leq n} \frac{ \norm[1]{\left\langle S(X)_{s,t}, \w \right\rangle}_{L^{\frac{p}{|\w|}}(\P)}}{|t-s|^{\alpha}} 
		\leq K \; \; \Rightarrow \;\; \norm[1]{\threebars X \threebars_{\alpha',n}}_{L^{\frac{p}{n}}(\Omega)} \leq C_1 (K + K^n).
	\end{equation}
	Similarly, given two processes $X$, $Y$ of bounded variation paths, for any  $p \geq 1$, $n \geq 1$,  $\alpha, \alpha' , \beta, \beta'> 0$ with $\alpha' < \alpha - \frac{1}{p}$, $\beta'< \beta - \frac{1}{p}$, there exists $C_2 >0 $, s.t. the bounds
	\begin{align*}
		&\sup_{0\leq s< t \leq T} \sup_{|\w|\leq n} \frac{ \norm[1]{\left\langle S(X)_{s,t}, \w \right\rangle}_{L^{\frac{p}{|\w|}}(\P)} + \norm[1]{\left\langle S(Y)_{s,t}, \w \right\rangle}_{L^{\frac{p}{|\w|}}(\P)}}{|t-s|^{\alpha |\w|}} \leq K, \\[0.5em]
		&\sup_{0\leq s< t \leq T} \sup_{|\w|\leq n} \frac{\norm[1]{\left\langle S(X)_{s,t}, \w \right\rangle - \left\langle S(Y)_{s,t}, \w \right\rangle}_{L^{\frac{p}{|\w|}}(\P)}}{|t-s|^{\alpha(|\w|-1)+\beta}} \leq \varepsilon,
	\end{align*}
	imply the estimate
	\begin{equation*}
		\left\| \rho_{\beta,\alpha,n}[ X, Y] \right\|_{L^{\frac{p}{n}}(\Omega)} \leq C_2 (1+K^{n-1}) \varepsilon.
	\end{equation*}
\end{proposition}

\begin{notation} \label{not:Nalphabeta}
	Given $0 < \beta \leq \alpha$, let $N_{\beta,\alpha}$ be the largest integer $n$ s.t. $\beta + (n-1) \alpha \leq 1$. We then let
	\[ \rho_{\beta,\alpha} := \rho_{\beta,\alpha,N_{\beta,\alpha}} \]
\end{notation}

This notation is consistent with usual rough path distances due to the following version of Lyons' extension theorem, which shows that the higher levels depend continuously on the first $N_{\beta,\alpha}$.

\begin{proposition} \label{prop:LExt}
	Let $0 < \beta \leq \alpha \leq 1$ and $N_{\beta,\alpha}$ be as in Notation~\ref{not:Nalphabeta}. Then for each $n \geq N_{\beta,\alpha}$, there exists a constant $C$ s.t.  for any two continuous paths of bounded variation $x,y :[0,T] \to \R^\d$, it holds that
	\begin{equation}
		\rho_{\beta,\alpha,n}[x;y] \leq C \left( 1 + \threebars x \threebars_{\alpha, n-N_{\beta,\alpha}} + \threebars y \threebars_{\alpha, n-N_{\beta,\alpha}} \right) \rho_{\beta,\alpha}[x;y]. 
	\end{equation}
\end{proposition}

\begin{corollary} \label{cor:BetaLE}
	Given $0 < \beta \leq \alpha$, there exists constants $C$ and $M$ s.t.  for any two continuous paths of bounded variation $x,y :[0,T] \to \R^\d$, it holds that
	\begin{equation}
		\rho_{\beta}[x;y] \leq C \left( 1 + \threebars x \threebars_{\alpha, M} + \threebars y \threebars_{\alpha, M} \right) \rho_{\beta,\alpha}[x;y]. 
	\end{equation}
\end{corollary}

\begin{remark} \label{rmk_alphabeta}
	The reason why we need to work with the $\beta \neq \alpha$ distances is precisely that in the applications we have in mind, such as obtaining optimal convergence rates of Gaussian rough paths, it may be that $\beta \ll \alpha$, in which case $N_{\beta,\beta}$ would be very large but $N_{\alpha,\beta}$ stays bounded, allowing us to only consider level 3 iterated integrals. This is similar to considerations in \cite{FR14}.
\end{remark}

\subsection{Main result} \label{sec_main_result}

\noindent
Let $X :[0,T] \to \R^\d$ be a continuous Gaussian process with \emph{independent} components~$X^{\i}$, $\i \in \llbracket \1,\d \rrbracket$, whose C--M space is~$\cH^{\i}$.
The C--M space of $X$ is then given by the direct sum ${\cH} = \oplus_{\i=1}^\d \cH^{\i}$, see~\cite[Example~2.3.8]{Bogachev}.
The following is our main result; we comment on the restriction on~$\alpha$ in Remark~\ref{rem:iid} below. 

\begin{theorem} \label{thm:cCYR1}
	Let $X$ be as above. Assume that for some~$q \geq 1$ and $\alpha > \frac{1}{4}$ with $\frac{1}{q}+ \alpha >1$, there exists $K>0$ s.t. for any $\i \in \llbracket \1,\d \rrbracket$, we have
	\begin{equation} \label{eq:ccyr}
		\forall\,  h \in \cH^{\i} \; \forall \, 0 \leq s \leq t \leq T: \quad \norm[0]{h}_{q-\var;[s,t]} \leq  K (t-s)^\alpha  \norm{h}_{\cH^{\i}},
	\end{equation}
	a condition we call \emph{controlled complementary Young regularity~(cCYR)}.
	Then, $X$ admits a canonical $\alpha^-$ H\"older rough path lift, with bounds only depending on the constant $K$ in the inequality.
\end{theorem}

\begin{proof}
	We assume that $X$ takes values in smooth ($1$-variation) paths. The general case follows by approximation, since the bounds we obtain are uniform.
	For any word $\w = \i_1 \ldots \i_n$ with length~$\abs[0]{\w} = n$, $n \in \N$, we aim to show that
	\begin{equation*}
		\left\| \sup_{0 \leq s<t \leq T} \frac{|\left\langle S(X)_{s,t}, \w \right\rangle|}{|t-s|^{n (\alpha -\varepsilon)}} \right\|_{L^r(\P)}^r \leq  C(K, \alpha,\varepsilon, n) \ \del[2]{1 + \E\sbr[2]{\threebars X \threebars_{\alpha-\eps,n-1}^{r(n-1) \vee 0}}}
	\end{equation*}
	for each $\varepsilon>0$ and each $r \geq 1$.
	In fact, by Kolmogorov's criterion (Proposition \ref{prop:Kolm}) we need only prove that for given $s<t$, the bound
	\begin{equation*}
		\norm[1]{\left\langle S(X)_{s,t}, \w \right\rangle}_{L^2(\P)}  \lesssim \abs[0]{t-s}^{(\alpha - \varepsilon) n}
	\end{equation*}
	for
	\begin{equation*}
		\left\langle S(X)_{s,t}, \w \right\rangle 
		:= \int_{s < r_1 < \ldots < r_n < t} \dif X^{\i_1}_{r_1} \dif X^{\i_2}_{r_2} \ldots \dif X^{\i_n}_{r_n}
	\end{equation*}
	holds, where $0 < \eps \ll 1$ is arbitrary. By Gaussian hypercontractivity in finite Wiener chaoses, the same bound then holds for all~$L^p(\P)$-norms, only with a different constant.
	
	We proceed by induction on the length~$\abs[0]{\w} = n$ and note that the case $n=0$ is trivial.
	Since~$\alpha > \frac{1}{4}$, we may also assume $n \leq 3$; the bounds for the higher levels are then automatic by Lyons' extension theorem, Proposition \ref{prop:LExt}.
	Thus, let~$n \in \{1,2,3\}$.	
	By the \emph{Poincar\'e inequality on Wiener space}, see Section~\ref{sec_notation}, we obtain the bound
	\begin{equation}
		\norm[1]{\left\langle S(X)_{s,t}, \w \right\rangle}_{L^2(\P)}
		\lesssim \, 
		\norm[3]{\sup_{\|h\|_{{\cH}} \leq 1} \abs[2]{\left\langle D \left\langle S(X)_{s,t}, \w \right\rangle, h\right\rangle_{{\cH}}}}_{L^2(\P)}
		+ \
		\abs[1]{\E\sbr[0]{\left\langle S(X)_{s,t}, \w \right\rangle}}.
		\label{eq:poincare}
	\end{equation}
	
	To estimate the first term on the RHS, note that for $h=(h^1, \ldots, h^d) \in {\cH}$, it holds by multilinearity that
	\begin{align*}
		\left\langle D \left\langle S(X)_{s,t}, \w \right\rangle, h \right\rangle_{{\cH}} &= \sum_{k=1}^n \int_s^t \dif X^{\i_1}_{r_1} \ldots \dif h^{\i_k}_{r_k}  \ldots \dif X^{\i_n}_{r_n} \\
		&\lesssim \del[1]{1+ \threebars X \threebars_{\alpha-\varepsilon,n-1}^{n-1}} (t-s)^{(n-1)(\alpha - \varepsilon)} \|h\|_{q-\var;[s,t]}
	\end{align*}
	by Young translation of rough paths (Proposition \ref{prop:YoungTr}). The induction hypothesis combined with \eqref{eq:ccyr} then yields the estimate
	\begin{equation*}
		\norm[3]{\sup_{\|h\|_{{\cH}} \leq 1} \abs[2]{\left\langle D \left\langle S(X)_{s,t}, \w \right\rangle, h\right\rangle_{{\cH}}}}_{L^2(\P)} \lesssim 
		C (t-s)^{n(\alpha-\varepsilon)}
	\end{equation*}
	
	We now consider the term $\E\sbr[0]{\left\langle S(X)_{s,t}, \w \right\rangle}$. By Wick's formula, this expectation is always zero when $n$ is odd, and (since we consider $n \leq 3$) only the case $n=2$ remains.
	Since~$X$ has independent components, the expectation is zero unless $\w=\i\i$ for some~$\i \in \llbracket 1, d \rrbracket$. 
	In that case, it is equal to $\frac{1}{2}\E[(X^i_{s,t})^2]$ which is bounded by (a constant times) $(t-s)^{2(\alpha-\varepsilon)}$ by the $n=1$ induction step, or from direct computations. 
\end{proof}

\begin{remark} \label{rem:iid}
	We emphasise that the argument to control the gradient term in~\eqref{eq:poincare} works for \emph{any}~$n \in \N$; 
	the restriction to $\alpha >\frac{1}{4}$ only comes from the need to control $\E\sbr[1]{\left\langle S(X)_{s,t}, \w \right\rangle}$ in the induction step above---which only works for $n<4$.
	However, in the case where the $X^{\i}$ are identically distributed, we can show that the result also holds for $n \leq 5$ which means we only need to assume $\alpha > \frac{1}{6}$. Indeed, it suffices to consider the case $\w=\i\i\j\j$ for~$\i \neq \j \in \llbracket \1,\d \rrbracket$, say~$\i = \1$ and~$\j = \2$ for concreteness.
	In that case, we can use the shuffle identity
	\begin{equation*}
		\1\1\2\2 =  \1\1\2 \shuffle \2 - \frac{1}{2} \1\1 \shuffle \2\2 + \frac{1}{2}(\2\2\1\1 - \1\1\2\2) + \frac{1}{2} (\2\1\2\1-\1\2\1\2) +\frac{1}{2} (\2\1\1\2-\1\2\2\1). 
	\end{equation*}
	The signature elements corresponding to the first two terms would have expectation of order $(t-s)^{4 \alpha}$ by the induction hypothesis for $n=1,2,3$ and Cauchy-Schwarz, while the last three have expectation zero by symmetry.
	
	We emphasise that, even in this case, we still crucially require the cCYR condition~\eqref{eq:ccyr}. 
	In particular, therefore, this remark does \emph{not} imply that one can construct a canonical lift of fBM with Hurst index~$H < 1/4$; in fact, condition~\eqref{eq:ccyr} is satisfied with $q := (1/2 + H)^{-1}$ and $\alpha := H$ in that case (see Corollary~\ref{coro:pl_approx} and the comments after it) and the condition~$\frac{1}{q} + \alpha > 1$ is equivalent to~$H > 1/4$.
\end{remark}

We now consider the distance between two Gaussian rough paths satisfying similar conditions as before.

\begin{theorem} \label{thm:cyrXX'}	
	Let $\widetilde{X}=(X,Y)$ be a Gaussian process in~$\R^{2d}$ with continuous sample paths such that both $X$ and~$Y$ have independent components, respectively.
	We assume that both $X$ and $Y$ satisfy the cCYR condition~\eqref{eq:ccyr} for the \emph{same}~$q \geq 1$ and $\alpha > 1/4$ with $1/q + \alpha > 1$ and the \emph{same} constant~$K$. 
	Finally, let $\widetilde{\cH}$ be the C--M space of $\widetilde{X}$ and further assume that there exists $\varepsilon>0$ s.t. for all  $\tilde{h} = (h,h') \in \widetilde{\cH}$ and for all $t < s$, the bound
	\begin{equation} \label{eq:ccyr2}
		\| h -  h'\|_{p-\var;[s,t]}  \leq \varepsilon  (t-s)^{\beta}  \| \tilde{h}\|_{\widetilde{\cH}},
	\end{equation}
	holds for some $p \geq 1$ and~$\beta \in (0, \alpha]$ such that
	\begin{equation}  \label{eq:ccyr2:cond}
		\frac{1}{p} + \alpha >1 , \;\;\;\; \frac{1}{q} + \beta >1,  \;\;\;\; \beta + 3 \alpha > 1. 
	\end{equation}
	Then, for all $\beta' < \beta$ and all $r \geq 1$, the bound
	\begin{equation}
		\norm[0]{\rho_{\beta{'}}[X,Y]}_{L^r(\P)} \leq C \varepsilon, 
	\end{equation}
	holds with a constant~$C = C(\alpha, q, K, \beta, p, \beta',r)$ that does not depend on $\varepsilon$.
\end{theorem}

\begin{proof}
	The proof has the same structure as the previous one. 
	First, note that by Theorem~\ref{thm:cCYR1}, we have that for any~$M \in \N$ the bounds
	\begin{equation*}
		\threebars X \threebars_{\alpha, M}\leq c, \quad \threebars Y \threebars_{\alpha, M} \leq c
	\end{equation*}
	hold with~$c = c(K)$.
	By Corollary \ref{cor:BetaLE} it is therefore enough to prove that
	\begin{equation}
		\norm[0]{\rho_{\beta',\alpha}[X,Y]}_{L^r(\P)} \equiv
		\norm[0]{\rho_{\beta',\alpha,N_{\beta',\alpha}}[X,Y]}_{L^r(\P)} \leq C \eps. 
		\label{thm:cyrXX':pf_distance_bd}
	\end{equation}
	Recall from Notation~\ref{not:Nalphabeta} that~$N_{\beta',\alpha}$ is the largest~$n \in \N$ such that~$\beta'  + (n-1) \alpha  \leq 1$;  by the assumption on $\beta$ in~\eqref{eq:ccyr2:cond}, we can choose $\beta'$ large enough to guarantee~$N_{\beta',\alpha} = 3$. Thus, it suffices to prove~\eqref{thm:cyrXX':pf_distance_bd} with~$N_{\beta',\alpha}$ replaced by~$n \leq 3$ and we proceed by induction on~$n$.
	By Proposition \ref{prop:Kolm}, it further suffices to prove that for a word~$\w$ of length $n$, we have
	\begin{equation} \label{eq:XX'}
		\norm[1]{\left\langle S(X)_{s,t} - S(Y)_{s,t} , \w \right\rangle}_{L^2(\P)}  \leq C \varepsilon |t-s|^{\beta  + (n-1) \alpha - \kappa} 
	\end{equation}
	for any $\kappa > 0$.
	
	Again by the \emph{Poincar\'e inequality}, it suffices to separately bound the $L^r(\P)$ norms of two terms. The first one is, for~$h \in \cH$ with $\norm[0]{h}_{\cH} \leq 1$, composed of terms of the form
	\begin{align*}
		& \int_s^t \dif X^{\i_1}_{r_1} \ldots \dif h^{\i_k}_{r_k}  \ldots \dif X^{\i_n}_{r_n} -  \int_s^t \dif Y^{\i_1}_{r_1} \ldots \dif h'^{\i_k}_{r_k}  \ldots \dif Y^{\i_n}_{r_n} \\
		= & \int_s^t \dif X^{\i_1}_{r_1} \ldots \dif h^{\i_k}_{r_k}  \ldots \dif X^{\i_n}_{r_n} -  \int_s^t \dif Y^{\i_1}_{r_1} \ldots \dif h^{\i_k}_{r_k}  \ldots \dif Y^{\i_n}_{r_n} \\ 
		+ &
		\int_s^t \dif X^{\i_1}_{r_1} \ldots \dif\, (h^{\i_k}_{r_k} - h'^{\i_k}_{r_k})  \ldots \dif X^{\i_n}_{r_n} \, .
	\end{align*} 
	By the Young translation bound (Proposition \ref{prop:YoungTr}), the first line is bounded by a multiple of
	\begin{equation*}
		(t-s)^{\alpha (n-2) + \beta}  \left( 1+ \threebars X \threebars_{\alpha,n-1}^{3}+ \threebars Y \threebars_{\alpha,n-1}^{3} \right)  \left\| h \right\|_{q-\var;[s,t]} \rho_{\beta,\alpha,n-1}[X,Y],
	\end{equation*}
	which, by the induction hypothesis and \eqref{eq:ccyr}, in turn has all its $L^r(\P)$ norms bounded by a multiple of $(t-s)^{\alpha(n-1) + \beta - \kappa} \eps$. The second line is estimated by
	\begin{equation*}
		\del[1]{1+ \threebars X \threebars_{\alpha,n}^{n-1}} (t-s)^{\alpha (n-1)}\norm[0]{h - h'}_{q-{\var};[s,t]}
	\end{equation*}
	which is again of the right order, in this case by~\eqref{eq:ccyr2}.
	The second term in the Poincar\'e inequality is 
	\begin{equation*}
		\E \sbr[1]{\left\langle S(X)_{s,t} - S(Y)_{s,t} , \w \right\rangle}
	\end{equation*}
	Since $n \leq 3$, as in the proof of Theorem~\ref{thm:cCYR1} this is only non-zero when $n=2$ and $\w = \i\i$ for some~$\i \in \llbracket 1,d \rrbracket$, in which case it is equal to
	\begin{align*}
		\frac{1}{2} \E\sbr[1]{(X^{\i}_{s,t})^2 - (Y^{\i}_{s,t})^2} 
		& \leq 
		\frac{1}{2} \E\sbr[1]{(X^{\i}_{s,t} - Y^{\i}_{s,t})^2}^{\frac{1}{2}}   
		\E\sbr[1]{(X^{\i}_{s,t})^2 + (Y^{\i}_{s,t})^2}^{\frac{1}{2}} \\
		& \lesssim C \eps (t-s)^{\alpha + \beta - \kappa}
	\end{align*}
	where the last estimate is due to the induction hypothesis.
\end{proof}

\begin{remark}
	By the same argument as in Remark \ref{rem:iid}, the condition in~\eqref{eq:ccyr2:cond} can be relaxed to $\beta + 5 \alpha > 1$ (from $\beta + 3 \alpha > 1$) if both $X$ and $Y$ have i.i.d. components.
	The same caveats as in Remark~\ref{rem:iid} apply.
\end{remark}

As an immediate corollary, we can estimate rough path distances in terms of suprema of expected distances. 

\begin{corollary} \label{cor:DistSup}
	Let $(X,Y)$ be a Gaussian process, where $X,Y$ satisfy \eqref{eq:ccyr} with the same parameters $\alpha > \frac{1}{4}$ and~$q \geq 1$ such that $\frac{1}{q} > 1 - \alpha$.
	Also, assume that~$X$ and~$Y$ have independent components, respectively.
	Then for any 
	$\beta'< \beta < \alpha$ with 
	\begin{equation*}
		\beta > (1-3 \alpha) \ \vee \ (1- 1/q) \ \vee \ \alpha (1-\alpha) q, 
	\end{equation*}
	it holds that 
	\begin{equation}
		\norm[1]{\rho_{\beta'}(X,Y)}_{L^r(\P)} \leq C \del[3]{\sup_{t \in [0,T]} \E \sbr[1]{\abs[0]{X_t-Y_t}}}^{1-\frac{\beta}{\alpha}}.
	\end{equation}
\end{corollary}

\begin{proof}
	First note that if  $\tilde{h} = (h,h') \in \widetilde{\cH}$, the C--M space of $\widetilde{X}$, by definition there exists a Gaussian random variable~$Z$ such that
	\begin{equation*}
		h_t = \E[Z X_t],\;\;\;h'_t = \E[Z Y_t],\;\;\;\; \norm[0]{\tilde{h}}_{\widetilde{\cH}}^2 = \E[Z^2],
	\end{equation*}
	and by Cauchy--Schwarz and the equivalence of $L^p(\P)$ norms for Gaussians, it holds that $| h_t - h'_t | \leq C \|\tilde{h}\|_{\widetilde{\cH}} \E \sbr[0]{\abs[0]{X_t - Y_t}}$.
	Also recall that for any $p = q/\theta$ with $\theta \in (0,1)$, by interpolation, for any $0 \leq s < t \leq T$ it holds that 
	\begin{align*}
		\| h -  h'\|_{p-\var;[s,t]}  &\leq C \| h -  h'\|_{q-\var;[s,t]}^{\theta} \left( \sup_{u \in [s,t]} |h_u-h'_u|\right)^{1-\theta}  \\
		& \leq  C  \norm[0]{\tilde{h}}_{\tilde{\cH}}^\theta (t-s)^{ \theta \alpha}  \norm[0]{\tilde{h}}_{\tilde{\cH}}^{1-\theta} \left(\sup_{t \in [0,T]} \E \sbr[0]{\abs[0]{X_t-Y_t}} \right)^{1-\theta}.
	\end{align*}
	where we have used the triangle inequality and~\eqref{eq:ccyr} to estimate $\norm[0]{h -  h'}_{q-\var;[s,t]}$.
	Then we apply Theorem \ref{thm:cyrXX'} with $\beta = \theta \alpha$, noting that $\frac{1}{p} + \alpha > 1$ is equivalent to $\beta >\alpha (1-\alpha) q$.
\end{proof}

\section{Applications} \label{sec_applications}

\noindent
In this section, we provide three applications of our main result.

\subsection{Convergence of approximations} \label{sec_kl_approx}

\noindent
Corollary~\ref{cor:DistSup} can be used to show convergence of standard approximations for Gaussian processes with continuous sample paths which are \emph{not} of bounded variation paths but satisfy \eqref{eq:ccyr}. 
We focus on the case of Karhunen--Loeve (KL) type and piecewise-linear (PL) approximations. For KL approximations, recall that for any orthonormal basis $(e_k)_{k \geq 0}$ of~$\cH$, there exist an i.i.d. sequence $(\gamma_k)_{k \geq 0}$ of standard Gaussians, s.t. we have
\begin{equation}
	X = \sum_{k \geq 0} \gamma_k e_k 
	\quad \text{in} \quad C([0,T],\R^d)
	\ \text{a.s. and in}~L^p(\P), \ p \in [1,\infty), \quad 
	\label{eq:kl_conv}
\end{equation}
see~\cite[App.~D.3]{friz-victoir-book}.
For~$N \in \N$, we set~$\mathcal{F}^N = \sigma(\gamma_k: k \leq N)$ and consider the approximation
\begin{equation}
	X^N := \E\sbr[0]{X \sVert[0] \mathcal{F}^N} = \sum_{k = 0}^N \gamma_k e_k
	\label{eq:kl_approx_N}
\end{equation}
For PL approximations, given a partition $D= \{ t_i \}$ of $[0,T]$, let $X^D$ be the piecewise-linear function which is equal to $X$ on grid-points and linearly interpolated in between, i.e.:
\begin{equation} \label{eq:defXD}
	X^D(t) = 
	X(t_i) + 
	\frac{t-t_i}{t_{i+1}-t_i} X_{t_i,t_{i+1}}, \;\;\; \text{for} \quad X_{t_i,t_{i+1}} := X(t_{i+1}) - X(t_i), \quad t \in [t_{i},t_{i+1}]. 
\end{equation}
We let $\abs[0]{D} := \sup_{i} \abs[0]{t_i - t_{i+1}}$ and then have the following convergence result.
\begin{corollary}
	Let $X$ be a Gaussian process with continuous sample paths satisfying \eqref{eq:ccyr}. Then there exists a rough path lift of $X$ such that, for any $\beta< \alpha$, it holds that
	\begin{equation*}
		\lim_{N \to \infty} \norm[0]{\rho_{\beta}(X,X^N)}_{L^r(\P)} = 0, \quad
		\lim_{|D| \to 0} \norm[0]{\rho_{\beta}(X,X^D)}_{L^r(\P)} = 0.
	\end{equation*}
\end{corollary}

\begin{proof}
	We show that both $X^N$ and $X^D$ form Cauchy sequences (in $L^r(\P)-\rho_{\beta}$ distances). By Corollary \ref{cor:DistSup}, it suffices to check that: (i) These processes converge in supremum expected distance. (ii) They satisfy the cCYR assumption \eqref{eq:ccyr} with a uniform constant.
	
	Regarding (i): Since \eqref{eq:ccyr} implies that $\E\sbr[0]{\abs[0]{X_t-X_s}} \leq C |t-s|^{\alpha}$, it follows that
	\begin{equation*}
		\sup_{t \in [0,T]} \E \sbr[1]{\abs[0]{X^D_t - X^{D'}_t}} \leq C \del[1]{\abs[0]{D}^\alpha \vee \abs[0]{D'}^{\alpha}} \to_{\abs[0]{D},\abs[0]{D'} \to 0} 0. 
	\end{equation*}
	Since $X^N$ is given by a conditional expectation, the statement for KL approximations immediately follows from~\eqref{eq:kl_conv}.
	
	Regarding (ii): The case of KL approximations is obvious, since their C--M spaces are $\cH^N = \operatorname{span}(e_k: k \leq N) \subset \cH$. For PL approximations, note that their C--M spaces are 
	\begin{equation*}
		\cH^D = \cbr[0]{h^D: \;\;h \in \cH} \quad \text{with} \quad \norm[0]{\tilde{h}}_{\cH^D} = \inf_{h \in \cH: \, h^D = \tilde{h}} \norm[0]{h}_{\cH},	
	\end{equation*}
	where, as in~\eqref{eq:defXD}, $h^D$ denotes the PL approximation of~$h$ on $D$. For $h \in \cH$, we let $\omega(s,t) := \norm[0]{h}_{q-\var;[s,t]}^q \leq C (t-s)^{q \alpha} \| h\|_{\cH}^q$. From the analysis in \cite[Prop.~5.20]{friz-victoir-book} , we know that for any $s<t$, we have
	\[ \|h^D \|_{q-\var;[s,t]} \leq \omega^D(s,t)^{1/q}, \]
	where, letting $D = \{t_i\}$, we have set
	\begin{equation*}
		\omega^D(s,t) := 
		\begin{cases}
			\left(\frac{t-s}{t_{i+1}-t_i}\right)^{q} \omega(t_i,t_{i+1}) & \text{if} \quad t_i \leq s \leq t \leq t_{i+1}, \\
			\omega^D(s,t_i) + \omega^D(t_i,t_j) + \omega^D(t_j,t) & \text{if} \quad s \leq t_i \leq t_j \leq t.
		\end{cases}
	\end{equation*}
	It is then immediate to check that this satisfies the bound
	\begin{equation*}
		\omega^D(s,t)  \leq C (t-s)^{\alpha p} \|h \|_{\cH}^p
	\end{equation*}
	for a uniform constant $C$.
\end{proof} 

It would be easy to derive convergence rates in the above corollary but they would typically be suboptimal. In the case where $\cH$ embeds into a suitable Besov space, a slightly refined analysis giving optimal rates is detailed in the next section.

\subsection{Convergence rates for PL approximations} \label{sec_pl_approx}

\noindent
Recall the definition of the Besov seminorm on~$[0,T]$ for $p \in [1,\infty)$ and $\delta \in (0,1)$:
\begin{equation*}
	\left\| f \right\|^p_{B^\delta_{p}([0,T])} := \int_{0 \leq s \leq t \leq T} \frac{ |f_t - f_s|^p}{|t-s|^{p \delta + 1}} ds \; dt.
\end{equation*}
An important case in which we can verify \eqref{eq:ccyr} is given by Friz--Victoir's Besov--variation embedding~\cite{FV06} which states that for fixed $\alpha >0$ and $p \in [1,\infty)$ with $\alpha + \frac{1}{p} \in (0,1)$, there exists $C>0$ such that, for any $s<t$, we have
\begin{equation}
	\left\| f \right\|_{(1/p+\alpha)^{-1}-\var,[s,t]} \leq C  (t-s)^{\alpha}\left\| f \right\|_{B^{1/p+\alpha}_{p}([s,t];\R^d)}.
	\label{eq:friz_victoir_embedding}
\end{equation}
It also turns out that in this case, we can obtain better estimates on convergence rates of (for instance) PL approximations. We start with the following approximation lemma.

\begin{lemma} \label{lem:BesovPart}
	For any $0< r '< r< 1$, there exists $C>0$ such that for any partition~$D$ of~$[0,T]$ and $f \in B^{r}_{p}([0,T])$, we have
	\begin{equation*}
		\norm[0]{f -f^{D}}_{B^{r'}_{p}([0,T])}\leq  C |D|^{r-r'} \left\|f \right\|_{B^{r}_{p}([0,T])}.
	\end{equation*}
\end{lemma}

\begin{proof}
	Fix $f \in B^{r}_{p}([0,T])$ and a partition $D=\{t_i\}$ of~$[0,T]$ with $\abs[0]{D} = \varepsilon$.
	Following \cite{FV06}, we note that an application of the Garsia--Rodemich--Rumsey lemma yields
	\begin{equation}
		|f_{s,t}| \leq \omega_f(s,t)^{\frac{1}{p}} |t-s|^{r-\frac{1}{p}} 
		\label{eq:grr_lemma}
	\end{equation}
	where $\omega_f(s,t) = C \| f \|_{\cB^{r}_{p}([s,t])}^p$ is a control function.
	In addition, it also holds that there exists a control $\omega_f^D(s,t)$ which satisfies $\omega_f^D(0,T) \leq \omega_f(0,T)$ and such that
	\begin{equation}
		|f^D_{s,t}| \leq 3^{1-1/p} \omega_f^D(s,t)^{\frac{1}{p}} |t-s|^{r-\frac{1}{p}}. 
		\label{eq:grr_lemma_2}
	\end{equation}
	To see this, we follow the proof of \cite[Prop. 5.20]{friz-victoir-book}. If $t_i \leq s \leq t \leq t_{i+1}$,we define
	\[ 
	\omega_f^D(s,t) = \left( \frac{t-s}{t_{i+1}-t_i} \right)^{p(1-r) +1} \omega_f(t_i,t_{i+1})  	\]
	and if $t_i \leq s \leq t_{i+1} \leq t_j \leq t$, we let
	\[
	\omega_f^D(s,t) = \omega_f^D(s,t_{i+1}) + \omega_f(t_{i+1}, t_j) +  \omega_f^D(t_j,t).
	\]
	The fact that $\omega_f^D$ is a control function is easily checked, using that $\omega_f$ is one and $p(1-r) +1 \geq 1$, and \eqref{eq:grr_lemma_2} follows from \eqref{eq:grr_lemma} and the definition of $f^D$.
	We will also use the elementary fact that, for any control function $\omega$, we have
	\begin{equation}
		\int_0^{T-h} \omega(t,t+h) \dif t  \leq  \omega(0,T) h.
		\label{eq:elemtary_control_estimate}
	\end{equation} 
	Indeed: first extend $\omega$ to $s,t \in [0,T+h]$ by letting $\omega(s,t)=\omega(s\wedge T, t \wedge T)$ (this preserves superadditivity). Then, letting $N= \lceil \frac{T}{h} \rceil$, it holds that, for all $0 \leq t \leq h$,
	\[
	\omega(t,t+h) +\omega(t+h,t+2h) \ldots + \omega(t+(N-1)h, t+ Nh) \leq \omega(0,T+h) = \omega(0,T)\]
	and \eqref{eq:elemtary_control_estimate} follows since
	\[
	\int_0^{T-h} \omega(t,t+h) \dif t \leq \sum_{j=0}^{N-1} \int_0^h  \omega(t+jh,t+(j+1)h) \dif t.
	\]

	We want to estimate
	\begin{equation*}
		\norm[0]{f-f^D}^p_{B^{r'}_{p,p}([0,T])} = \int_{0 \leq s \leq t \leq T} \frac{ |f_{s,t} - f^D_{s,t}|^p}{|t-s|^{p r' + 1}} \dif s \; \dif t, \quad f_{s,t} = f_t - f_s,
	\end{equation*}
	which we will achieve by splitting the integral depending on whether $|t-s|$ is smaller or greater than $\varepsilon$.
	
	In the first case, we find 
	\begin{align*} 
		&
		\int_0^{\eps} \int_0^{T-h} \abs[0]{f_{t,t+h} - f^D_{t,t+h}}^{p} \dif t \, \frac{1}{h^{pr'+1}} \dif h \notag\\
		\lesssim \ 
		& \int_0^{\eps} \int_0^{T-h} \del[1]{\omega_f(t,t+h) + \omega_f^D(t,t+h)} \dif t  \, \frac{h^{rp - 1} }{h^{pr'+1}} \dif h \label{lem:BesovPart:pf_first_part} \\
		\lesssim \ 
		& \omega_f(0,T) \int_0^{\eps} h^{(r-r') p - 1} \dif h 
		\, \lesssim \, \omega_f(0,T) \varepsilon^{p(r-r')}. \notag
	\end{align*}
	where we have used~\eqref{eq:grr_lemma} and~\eqref{eq:grr_lemma_2} in the first and~\eqref{eq:elemtary_control_estimate} in the second estimate.
	
	For the second case, recall that~$\abs[0]{D} = \eps$. 
	Next, note that if $s, t \in [0,T]$ with~$s \leq t$ and~$\abs[0]{t - s} > \eps$, there are~$t_i, t_j \in D$ such that~$t_i$ (resp.~$t_j$) is at distance smaller than $\eps$ from $s$ (resp. $t$) and s.t. $s \leq t_i \leq t_j \leq t$.
	Using that~$f_r^D = f_r$ for any~$r \in D$, we then find
	\begin{align*}
		|f_{s,t} - f^D_{s,t}| &=  |f_{s,t_i} - f^D_{s,t_i} + f_{t_j,t} - f^D_{t_j,t} | \\
		& \leq \varepsilon^{r - \frac{1}{p}} \del[1]{\omega_f(s,s+\varepsilon)^{\frac{1}{p}} + \omega_f^D(s,s+\varepsilon)^{\frac{1}{p}} + \omega_f(t-\varepsilon,t)^{\frac{1}{p}} + \omega_f^D(t-\varepsilon,t)^{\frac{1}{p}}}.
	\end{align*}
	where we have again used~\eqref{eq:grr_lemma} and~\eqref{eq:grr_lemma_2} in the estimate.
	Integrating and~\eqref{eq:elemtary_control_estimate} gives
	\begin{equation*}
		\int_{|s-t| \geq \varepsilon} \frac{ |f_{s,t} - f^D_{s,t}|^p}{|t-s|^{p r' + 1}} \dif s \dif t \lesssim \omega_f(0,T) \varepsilon^{rp} \int_{ \varepsilon}^T h^{-r'p - 1}  \dif h \lesssim \omega_f(0,T) \varepsilon^{p(r-r')}
	\end{equation*}
	which concludes the proof.
\end{proof}	

This gives the following result on convergence of PL approximations.

\begin{corollary} \label{coro:pl_approx}
	Let $X$ be a Gaussian process with independent components and con\-ti\-nuous sample paths, and assume that its C--M space satisfies
	\begin{equation*}
		\cH \subset B_p^{1/p + \alpha}, \quad \alpha > \frac{1}{4}, \quad 2 \alpha + \frac{1}{p} > 1.
	\end{equation*} 
	Then $X$ admits a canonical $\alpha^-$ rough path lift, and for any $\beta' < \beta \leq \alpha$ with $\alpha + \beta + \frac{1}{p} > 1$, $\beta + 3 \alpha > 1$, and any $r \geq 1$, it holds that, for some $C>0$, we have
	\begin{equation*}
		\norm[0]{\rho_{\beta'}(X,X^D)}_{L^r(\P)} \leq   C |D|^{\alpha-\beta}. 
	\end{equation*}
\end{corollary}

\begin{proof}
	Lemma \ref{lem:BesovPart} combined with Friz--Victoir variation embedding~\eqref{eq:friz_victoir_embedding} gives
	\begin{equation*}
		\norm[1]{f - f^D}_{(1/p+\beta)^{-1}-\var,[s,t]} 
		\leq C  \abs[0]{D}^{\alpha-\beta} (t-s)^{\beta}\left\| f \right\|_{B^{1/p+\alpha}_{p}([s,t])} \leq C' \abs[0]{D}^{\alpha-\beta} (t-s)^{\beta}\left\| f\right\|_{\cH},
	\end{equation*}
	so that the result is a direct application of Theorem \ref{thm:cyrXX'}.
\end{proof}

Note that the highest convergence rate (in the lowest regularity) is obtained by letting $\beta \to 1-1/p - \alpha$, with a rate of $2 \alpha - 1 + 1/p$.

The important case of $p=2$ corresponds to C--M spaces of fractional Brownian motions, in which case we recover the optimal rate of $2 \alpha - 1/2$ from Friz-Riedel \cite{FR14}.

\subsection{Rough paths bounds for random Fourier series} \label{sec_random_fourier}

\noindent
Similar considerations as in the last subsection also allow to get rough path bounds on random Fourier series with minimal assumptions, which we now detail. Let $(e_k)_{ k \geq 1}$ be the usual trigonometric basis of functions on $[0, 2\pi]$, defined by
\[  e_{2k}(t) = \cos( k t), \;\; e_{2k+1}(t) = \sin(k t ), \;\; k \geq 1, \;\; t \in [0,2\pi] \]

We will use the fact that Besov norms of Fourier series admit simple expressions in terms of the coefficients when~$p = 2$ since~$B_2^\alpha([0,2\pi]) \simeq H^\alpha([0,2\pi])$.
More precisely, for
\begin{equation*}
	f = \sum_{k \geq 1} f_k e_k
\end{equation*}
with $(f_k)_{k \geq 1} \in \ell^2$, it holds that for some $C>0$, we have the estimate
\begin{equation} \label{eq:BesovFourier}
	C^{-1} \sum_{k \geq 1} f_k^2 k^{2 \alpha}
	\leq  \norm[0]{f}_{\cB_2^{\alpha}([0,2\pi])}^2 
	\leq C \sum_{k \geq 1} f_k^2 k^{2 \alpha}
\end{equation}
see, e.g., \cite{ST87}.
Combined with previous observations, we obtain the following result:

\begin{proposition} \label{prop:random_fourier_series}
	Let~$X = (X^\1,\ldots,X^\d)$ and~$Y = (Y^\1,\ldots,Y^\d)$ be random vectors s.t.
	\begin{equation*}
		X^{\i}(t) = \sum_{k} x_k^{\i} \gamma_k^{\i} e_k(t), \quad
		Y^{\i}(t) = \sum_{k} y_k^{\i} \gamma_k^{\i} e_k(t), \quad
		\i \in \llbracket 1,d \rrbracket,
	\end{equation*}
	where
	\begin{equation*}
		\{\gamma_k^{\i}: k \geq 1, \i \in \llbracket 1,d \rrbracket\}
	\end{equation*}
	is a family of i.i.d. $\cN(0,1)$ random variables, and, for each~$i \in \llbracket 1,d \rrbracket$,  $(x_k^{\i})_{k \geq 1}$, $(y_k^{\i})_{k \geq 1}$ are deterministic sequences such that for some $\alpha > 1/4$, we have
	\begin{equation}
		K := \max_{\i \in \llbracket \1,\d \rrbracket} K^{\i} < \infty, \quad
		K^{\i} := \sup_k \del[1]{\abs[0]{x^{\i}_k} + \abs[0]{y_k^{\i}}} k^{\frac{1}{2}+\alpha}.
		\label{prop:random_fourier_series:coeff_condition}
	\end{equation}
	Then $X$ and $Y$ both admit an $\alpha^-$ H\"older rough path lift and for any $\beta'< \beta \leq\alpha$ with $\beta + \alpha > \frac 1 2$, it holds that
	\begin{equation} \label{eq:HolderFourier}
		\rho_{\beta'}(X,Y) \leq C  \max_{i \in \llbracket \1,\d \rrbracket} \sup_{k \geq 1}  \del[1]{\abs[0]{x_k^{\i} - y_k^{\i}} k^{\frac{1}{2}+\beta}},
	\end{equation}
	where $C = C(\alpha, \beta, K, \alpha')$.	
\end{proposition}

\begin{proof}
	Note that~$\cH^X$ and~$\cH^Y$, the respective C--M spaces of $X$ and~$Y$,  admit the direct representation
	\begin{align*}
		\cH^{X^{\i}} 
		& = \cbr[3]{h^{X^\i} =\sum_{k \geq 1} x_k^{\i} u_k^{\i} e_k, \;\: \sum_{k\geq 1} \del[1]{u^{\i}_k}^2 < \infty, \; \: \norm[0]{h^{\i}}_{\cH} = \| u^{\i}\|_{\ell^2}}, \quad
		\cH^X = \bigoplus_{i=1}^{d} \cH^{X^{\i}},
		\\ 
		\cH^{Y^{\i}} 
		& = \cbr[3]{h^{Y^\i} =\sum_{k \geq 1} y^{\i}_k u^{\i}_k e_k, \;\: \sum_{k\geq 1} \del[1]{u^{\i}_k}^2 < \infty, \; \: \norm[0]{h^{\i}}_{\cH} = \| u^{\i}\|_{\ell^2}}, \quad
		\cH^Y = \bigoplus_{i=1}^{d} \cH^{Y^{\i}}.
	\end{align*}
	In particular, the assumption~\eqref{prop:random_fourier_series:coeff_condition} on the coefficients implies that
	\begin{equation*}
		\cH^X, \cH^Y \subset \cB_2^{\alpha + \frac{1}{2}}([0,2\pi];\R^d)
	\end{equation*}
	which, by Corollary~\ref{coro:pl_approx}, means that $X$ and $Y$ each admit $\alpha^-$ Hölder rough path lifts.
	
	Let~$\cH^{\i} := \cH^{(X^\i, Y^\i)}$. 
	For the distance bound, note that for any element
	\begin{equation*}
		h = (h^X,h^Y) = \del[1]{h^{X^\1}, \ldots, h^{X^\d}, h^{Y^\1}, \ldots, h^{Y^\d}} \in \cH^{(X,Y)},
	\end{equation*}
	by definition of the C--M space we can write
	\begin{equation*}
		h^{\i} 
		:= (h^{X^\i}, h^{Y^\i}) = \sum_{k \geq 1} u_k^{\i} 
		\begin{pmatrix}
			x_k^{\i} \\
			y_k^{\i}
		\end{pmatrix}
		e_k,
		\quad
		\norm[0]{h^{\i}}_{\cH^\i} = \norm[0]{u^\i}_{\ell^2},
		\quad
		\i \in \llbracket 1,d \rrbracket. 
	\end{equation*}
	Combining this observation with~\eqref{eq:BesovFourier} and the Friz--Victoir variation embedding in~\eqref{eq:friz_victoir_embedding}, we arrive at
	\begin{align*}
		& \thinspace
		\norm[0]{h^X - h^Y}_{(\beta+ \frac{1}{2})^{-1}-\var;[s,t]} \\
		\leq & \
		C (t - s)^{\beta} \max_{\i \in \llbracket \1,\d \rrbracket} \norm[0]{h^{X^\i} - h^{Y^\i}}_{B_2^{1/2+\beta}([s,t])} \\
		\lesssim & \
		(t - s)^{\beta} \max_{\i \in \llbracket \1,\d \rrbracket}
		\del[3]{\sum_{k \geq 1} \del[0]{u_k^\i}^2 \del[0]{x_k^\i - y_k^\i}^2 k^{1+2 \beta}}^{\frac{1}{2}} \\
		\lesssim & \
		(t - s)^{\beta} \max_{\i \in \llbracket \1,\d \rrbracket} \sup_{k \geq 1} \abs[1]{\del[0]{x_k^\i - y_k^\i} k^{\frac{1}{2}+\beta}} \norm[0]{u^\i}_{\ell^2} \\
		= & \
		(t - s)^{\beta} \max_{\i \in \llbracket \1,\d \rrbracket} \sup_{k \geq 1} \abs[1]{\del[0]{x_k^\i - y_k^\i} k^{\frac{1}{2}+\beta}} \norm[0]{h^\i}_{\cH^\i} \\
		\leq & \
		(t - s)^{\beta} \max_{\i \in \llbracket \1,\d \rrbracket} \sup_{k \geq 1} \abs[1]{\del[0]{x_k^\i - y_k^\i} k^{\frac{1}{2}+\beta}} \norm[0]{h}_{\cH}
	\end{align*}
	where the RHS is finite by assumption~\eqref{prop:random_fourier_series:coeff_condition}.
	Again, we conclude by Theorem~\ref{thm:cyrXX'}.
\end{proof}

\noindent	
In particular, if $X^N$ is the projection of $X$ onto the Fourier modes that are less than or equal to~$N$, then $X^N \to X$ as~$N \to \infty$ in $\beta$-rough path topology at a rate of almost $N^{\alpha-\beta}$.

Note that a rough path estimate of the form \eqref{eq:HolderFourier} is the best one can get, in the sense that the quantity on the RHS already gives a sharp estimate for $\beta$-H\"older bounds at level~$1$ (see, e.g., \cite{K85}). 
Our result improves previously known results on Fourier series from \cite{FGGR16} which required stronger assumptions on the coefficients $x_k, y_k$. 
This highlights that, in some contexts, C--M regularity may be more easy to check than estimates based on the $2$D variation regularity of the covariance function.

\paragraph*{Acknowledgements.} TK gratefully acknowledges funding via Giuseppe Cannizzaro's EPSRC grant "Large scale universal behaviour of Random Interfaces and Stochastic Operators." This work was largely written while TK was employed at TU Berlin and finished when he was based at the University of Warwick. Both authors thank the Université Paris-Dauphine resp. TU Berlin for their hospitality and financial support during mutual visits. 
We thank the referees for their careful reading of our manuscript and their valuable feedback.

\bibliographystyle{plainalpha}
\bibliography{refs.bib}

\end{document}